\documentclass[12pt,reqno]{amsart}
\usepackage[utf8]{inputenc}
\usepackage{amssymb,amsfonts,amsmath,amsthm}
\usepackage{mathtools}
\usepackage{indentfirst}
\usepackage{bigints}
\usepackage[letterpaper, margin=1.15 in, footskip=.4in]{geometry}
\usepackage{xcolor}
\usepackage{nccmath}
\usepackage{geometry}

\makeatletter
\def\specialsection{\@startsection{section}{1}
  \z@{\linespacing\@plus\linespacing}{.5\linespacing}
  {\normalfont}}
\def\section{\@startsection{section}{1}
  \z@{.7\linespacing\@plus\linespacing}{.5\linespacing}
  {\normalfont\scshape\bfseries}}
\makeatother

\usepackage{fullpage}
\usepackage{fancyhdr}

\usepackage{stmaryrd}
\usepackage{mathrsfs}
\usepackage{hyperref}
\usepackage{lineno}

\usepackage{tikz}
\usetikzlibrary{calc}
\usepackage{caption}
\usepackage{lmodern}

 \newtheorem{theorem}{Theorem}[section]
 \newtheorem{lemma}[theorem]{Lemma}
 \newtheorem{corollary}[theorem]{Corollary}
 \newtheorem{prop}[theorem]{Proposition}
 \newtheorem*{remark}{Remark}

 \let\originalleft\left
\let\originalright\right
\renewcommand{\left}{\mathopen{}\mathclose\bgroup\originalleft}
\renewcommand{\right}{\aftergroup\egroup\originalright}

\newcommand{\Mod}[1]{\ (\mathrm{mod}\ #1)}

\title{Partitions into sums of two squares}
\author{Jaime Palacios}

\newcommand{\RN}[1]{%
  \textup{\uppercase\expandafter{\romannumeral#1}}%
  }

\begin{document}

\begin{abstract}
     We use a variation of the Circle Method, along with the Saddle Point Method, to obtain an asymptotic formula for the number of partitions of a number $n$ into integers which are sums of two squares. Unlike previous work on partitions into restricted parts, we need to handle a Dirichlet series with a fractional singularity.
\end{abstract}

\maketitle

\section{Introduction}

 A \textit{partition} of a number $n$ is a non-increasing sequence of positive integers whose sum equals $n$, and the count   of such is denoted by the \textit{partition function} $p(n)$. 
The study of asymptotics of the partition function started over a century ago with a result of Hardy and Ramanujan from 1918 \cite{hardy} in which they showed
 \begin{equation*}
     p(n)\sim \frac{1}{4n\sqrt{3}}\exp\Big(\pi \sqrt{\frac{2}{3}}n^{\frac{1}{2}}\Big),
 \end{equation*}
 as $n\to \infty$. Their method was soon after adjusted by Hardy and Littlewood for the Waring's problem, being now referred to as the Hardy-Ramanujan-Littlewood Circle Method.

 One may also consider the partitions of $n$ whose parts all lie in a subset $\mathcal{A}$ of the integers, denoting by $p_\mathcal{A}(n)$ the \textit{restricted partition function}. Denote by $\mathcal{S}$ the set of sums of two squares, and its restricted partition function by $p_\mathcal{S}(n)$. The main result in the present paper (Theorem \ref{main}) is an asymptotic formula for $p_\mathcal{S}(n)$. This is given in terms of quite complicated auxiliary functions. However, it can be simplified to the following asymptotic estimate:
 
 \begin{theorem}\label{1}
 The number of partitions of $n$ with all parts lying in the set $\mathcal{S}$ of sums of two squares satisfy
    \begin{equation*}
        p_\mathcal{S}(n)\sim 2^{-9/8}\Big(\frac{K}{3}\Big)^{1/4}n^{-\frac{3}{4}}(\log n)^{-\frac{1}{8}} \exp\Bigg(2^{3/4}\pi\sqrt{\frac{K}{3}}n^{1/2}(\log n)^{-1/4}\Big(1+o(1)\Big)\Bigg),
    \end{equation*}
    as $n\to \infty$, and where $K$ is the Landau-Ramanujan constant (see Theorem \ref{Landau} for the definition of such constant).
\end{theorem}
 
The Circle Method starts by expressing the partition function as an integral applying Cauchy's integral formula to the generating series. This integral is typically estimated by breaking down the interval of integration into two types of subintervals: the \textit{major arcs}, composed of intervals close to a rational number with small denominator, which constitute the main term of the integral; and the \textit{minor arcs}, which are the complement of the major arcs and which make up the error term. This decomposition is used in other contexts where an integral of an exponential appears.

Our proof follows an elegant variation of the Circle Method devised by Vaughan \cite{vaughan}, which makes use of the Saddle Point Method, for obtaining the asymptotic formula for the partitions with all parts being prime numbers. This method has also been used for enumerating the restricted partitions for other sets such as  prime powers (Gafni \cite{gafni2}), semiprimes (Das, Robles, Zaharescu, and Zeindler \cite{robles}), perfect powers (Gafni \cite{gafni1}), powers in residue classes (Berndt, Malik, and Zaharescu \cite{malik}), and polynomial values (Dunn and Robles \cite{dunn}).  All of these, our case of sums of two squares included, have their main term coming only from the major arc around zero, which we will call \textit{principal} following Gafni's \cite{gafni2} terminology, and the rest of the major arcs, called consequently \textit{nonprincipal}, contribute to the error term just as the minor arcs. 

Another result within the same framework is due to Debruyne and Tenenbaum \cite{dt}. They deal with some class of restricted partition functions whose Dirichlet series can be meromorphically continued into the left half complex plane.
The perfect power, power in residue class, and polynomial value cases satisfy this condition and follow consequently. However, the prime, prime power, and semiprime cases do not: their Dirichlet series contain logarithmic singularities. 

For the sums of two squares a singularity is present as well, which prevents its Dirichlet series to be meromorphically continued into the left half complex plane. This time it is a fractional singularity, making things a bit trickier. Thus we must deal with the Dirichlet series by means different to 
the ones used in previous work containing singularities. We also exploit an advantage given in \cite{dt}, which spares us from having to estimate the exponential sums for the nonprincipal major arcs nor for the minor arcs one encounters when working with the Circle Method. As a result, we obtain a simpler and shorter proof.

Our paper's layout is as follows. In Section \ref{gs} we write our restricted partition function as an integral of an exponential using the generating series. Section \ref{results} presents our results. We then continue with estimates for $\Big(\rho \frac{d}{d\rho}\Big)^m \Phi(\rho)$ in Section \ref{operator}, where the Dirichlet series with fractional singularity appears. In Section \ref{error} we estimate the contribution coming from what can be regarded as the nonprincipal major arcs and the minor arcs.
We conclude the proof of our main theorem in Section \ref{proofmain}, along with the proof of the difference function in Section \ref{diff}. Finally, Section \ref{appendix} is an appendix with some well known results we refer to and use throughout the paper for the convenience of the reader.

\section{Generating Series} \label{gs}

Let $\mathcal{S}:= \{\ell \in \mathbb{N}: \ell=x^2+y^2, x,y\in\mathbb{N}\}$ and denote by $p_\mathcal{S}(n)$ the functions that counts the number of partitions of $n$ whose parts are elements of the set $\mathcal{S}$. 
We start with its generating series,
\begin{equation*}
    \Psi(z)=\sum_{n=1}^\infty p_\mathcal{S}(n) z^n = \prod_{\ell \in \mathcal{S}} \frac{1}{1-z^\ell}. 
\end{equation*}

It will be more convenient for us to write it as
\begin{equation*}
    \Psi(z)=\exp(\Phi(z)),
\end{equation*}
where
\begin{equation*}
    \Phi(z)=\sum_{j=1}^\infty \sum_{\ell \in \mathcal{S}}\frac{1}{j} z^{j\ell}. 
\end{equation*}
Expressing the generating series as an exponential will make sense when we make use of the Saddle Point Method.

Suppose  $0<\rho<1$, then by Cauchy's integral formula we have
\begin{equation}\label{Cauchy}
    p_\mathcal{S}(n)=\rho^{-n}\int_{-1/2}^{1/2} \Psi(\rho e(\alpha)) e(-n\alpha) d\alpha = \rho^{-n} \int_{-1/2}^{1/2} \exp(\Phi(\rho e(\alpha))) e(-n\alpha) d\alpha,
\end{equation}
for any positive real number $\rho<1$, and where $e(\alpha)=e^{2\pi i \alpha}$ as usual. 

Let $x\in \mathbb{R}$ be large (eventually we will set $x=n$). We choose $\rho=\rho(x)$ so that
\begin{equation}\label{def}
    x=\rho \Phi'(\rho).
\end{equation}
The point $\rho$ corresponds to what is called the \textit{saddle point}, which will be important when the Saddle Point Method makes its presence. And it will follow from Lemma \ref{deriv} that the relationship between $x$ and $\rho$ is well defined and injective, and that $\rho \to 1^-$ as $x \to \infty$.

\section{Asymptotic Formula} \label{results}
We now state our main theorem: the asymptotic formula for the partition function $p_\mathcal{S}(n)$.

\begin{theorem}\label{main}
Let $\rho=\rho(n)$, $\Psi(\rho)$ and $\Phi(\rho)$ be define as above, then as $n\to \infty$ our partition function is
    \begin{equation*}
        p_\mathcal{S}(n)=\frac{\rho^{-n}\Psi(\rho)}{\sqrt{2 \pi \Big\{\Big(\rho\frac{d}{d\rho}\Big)^2 \Phi(\rho)\Big\}}}\Big(1+O(n^{-\frac{1}{5}})\Big).
    \end{equation*}
\end{theorem}

The following proposition will allow us to get an asymptotic estimate in terms of $n$ for our main theorem, namely Theorem \ref{1}.

\begin{prop}\label{p}
As $x\to \infty$, we have
    \begin{equation} \label{p1}
        x \log \frac{1}{\rho(x)}=\pi\sqrt{\frac{K}{3}} x^{1/2}(2\log x)^{-1/4}\Big(1-\frac{1}{8}\frac{\log \log x}{\log x}+O\big(\frac{1}{\log x}\big)\Big),
    \end{equation}
   
    \begin{equation}\label{p2}
        \Phi\big(\rho(x)\big)=\pi\sqrt{\frac{K}{3}} x^{1/2}(2\log x)^{-1/4}\Big(1-\frac{1}{8}\frac{\log \log x}{\log x}+O\big(\frac{1}{\log x}\big)\Big),
    \end{equation}
    
    and for $m\geq 1$
    
    \begin{equation}\label{p3}
        \Big(\rho \frac{d}{d\rho}\Big)^m \Phi(\rho(x))=x^{\frac{m+1}{2}}\Bigg(\frac{3\sqrt{2 \log x}}{K\pi^2}\Bigg)^{\frac{m-1}{2}}\Gamma(m+1) \Bigg(1+O\Big(\frac{\log \log x}{\log x}\Big)\Bigg).
    \end{equation}
\end{prop}

\begin{proof}
    Suppose $x$ is sufficiently large. Then $\rho$ is close to $1$ and so $X$ is large. From the definition of $\rho$ given by equation (\ref{def}), Lemma \ref{deriv}, and  $\zeta(2)=\frac{\pi^2}{6}$, we see that
    \begin{equation}\label{x}
        x=\rho\Phi'(\rho)=\frac{K\pi^2X^2}{6(\log X)^{1/2}}\Big(1+O\Big(\frac{1}{\log X}\Big)\Big).
    \end{equation}
    Thus $$\log x=2\log X-\frac{1}{2}\log \log X+O(1).$$
    Whence $\log x\ll \log X\ll \log x$, so it must be  that $\log \log X=\log \log x+O(1)$. Then,
    $$\log X=\frac{1}{2}\log x+\frac{1}{4}\log \log x +O(1).$$
    Putting this into equation (\ref{x}) and solving for $X$ we obtain
    \begin{equation}\label{X}
        X=\pi^{-1}\sqrt{\frac{3}{K}}x^{1/2}(2\log x)^{1/4}\Big(1+\frac{1}{8}\frac{\log \log x}{\log x}+O\big(\frac{1}{\log x}\big)\Big),
    \end{equation}
    which implies (\ref{p1}) upon noticing that $ x\log \frac{1}{\rho}=xX^{-1}$.

    Now, combining (\ref{p1}) with Lemma \ref{deriv} we get

    \begin{equation*}
        \Big(\rho \frac{d}{d\rho}\Big)^m \Phi(\rho)=xX^{m-1}\Gamma(m+1)\Big(1+O\Big(\frac{1}{\log X}\Big)\Big).
    \end{equation*}
    Plugging equation (\ref{X}) into this gives (\ref{p2}) when $m=0$ and (\ref{p3}) when $m\geq 1$.

\end{proof}

We may now give the asymptotic estimate of Theorem \ref{main}.

\begin{proof}[Proof of Theorem \ref{1} given Theorem \ref{main}. ]
    By Proposition \ref{p} we have
    \begin{align}
    \label{phi}
    \begin{split}
        \rho^{-n}\Psi(\rho)&=\exp\big(n \log \frac{1}{\rho(n)}+\Phi(\rho(n))\big)\\
        &= \exp\Bigg(2^{3/4}\pi\sqrt{\frac{K}{3}}n^{1/2}(\log n)^{-1/4}\big(1+o(1)\big)\Bigg)
    \end{split}
    \end{align}
    and
    \begin{align*}
        \sqrt{ \Big\{\Big(\rho\frac{d}{d\rho}\Big)^2 \Phi(\rho)\Big\}}&=\Bigg(n^{3/2}\Big(\frac{3\sqrt{2\log n}}{K\pi^2}\Big)^{1/2}\Gamma(3)\Big(1+O\Big(\frac{\log \log n}{n}\Big)\Big)\Bigg)^{1/2}\\
        &= K^{-1/4}2^{5/8}3^{1/4}\pi^{-1/2}n^{3/4}(\log n)^{1/8}\big(1+o(1)\big).
    \end{align*}
    Therefore, combining these with Theorem \ref{main} gives the desired result.
\end{proof}

The method used for asymptotic formula in Theorem \ref{main} can also be used for the estimate of the difference function of $p_{\mathcal{S}}(n)$, provided in the following theorem.
\begin{theorem}\label{growth}
Using the notation defined above with $\rho=\rho(n)$, we have
    \begin{equation*}
        p_\mathcal{S}(n+1)-p_\mathcal{S}(n)=\frac{\rho^{-n}\log(\frac{1}{\rho})\Psi(\rho)}{\sqrt{2 \pi  \Big\{\Big(\rho\frac{d}{d\rho}\Big)^2 \Phi(\rho)\Big\}}}\Big(1+O(n^{-\frac{1}{5}})\Big),
    \end{equation*}
    as $n \to \infty$.
\end{theorem}

From Theorem \ref{growth} and Proposition \ref{p} one may also deduce an asymptotic estimate for the growth of our restricted partition function, which takes the following form.

\begin{corollary}
Let $\rho(n)$ be defined as above. Then
    \begin{equation*}
        p_\mathcal{S}(n+1)-p_\mathcal{S}(n)\sim \pi\sqrt{\frac{K}{3}}n^{-1/2}(2\log n)^{-1/4}p_S(n),
    \end{equation*}
    as $n\to \infty$.
\end{corollary}

\begin{proof}
    From Theorems \ref{main} and \ref{growth} we may see that
    \begin{equation*}
        p_\mathcal{S}(n+1)-p_\mathcal{S}(n)=p_\mathcal{S}(n)\log\Big(\frac{1}{\rho}\Big)  \Big(1+O(n^{-1/5})\Big).
    \end{equation*}
    And from Proposition \ref{p} we get that
    \begin{equation*}
        \log\Big(\frac{1}{\rho}\Big)=\pi\sqrt{\frac{K}{3}} n^{-1/2}(2\log n)^{-1/4}\Big(1-\frac{1}{8}\frac{\log \log n}{\log n}+O\big(\frac{1}{\log n}\big)\Big).
    \end{equation*}
    Combining these two equations gives us the corollary.
\end{proof}

\section{Estimates for \texorpdfstring{$\Big(\rho \frac{d}{d\rho}\Big)^m \Phi(\rho)$}{a}} \label{operator}

Now we obtain estimates for $\Big(\rho \frac{d}{d\rho}\Big)^m \Phi(\rho)$, which will correspond to the derivatives of a composition of functions whose Taylor expansion will be used later. This will become clear when we resort to the Saddle Point Method. 

The proof of the lemma below follows the spirit of that of Landau's well known result about the number of sums of two squares  (Theorem \ref{Landau}).

\begin{lemma}\label{deriv} Let $\rho=e^{-1/X}$, then
    \begin{equation} 
        \Big(\rho \frac{d}{d\rho}\Big)^m \Phi(\rho)= \frac{K \zeta(2) \Gamma(m+1)X^{m+1}}{\sqrt{\log X}}\Bigg(1+O\Big(\frac{1}{\log X}\Big)\Bigg),
    \end{equation}
    as $\rho \to 1^-$.
\end{lemma}

\begin{proof}
We have that
\begin{align*}
    \Big(\rho \frac{d}{d\rho}\Big)^m\Phi(\rho)&=\sum_{j=1}^\infty \sum_{\ell\in \mathcal{S}} \frac{1}{j} (j\ell)^m \rho^{j\ell}\\
    & =\sum_{j=1}^\infty \sum_{\ell\in \mathcal{S}} \frac{1}{j} (j\ell)^m e^{-j\ell/X}.
\end{align*}
   Using an inverse Mellin transform and interchanging the sums with the integrals, we obtain
\begin{equation*}
     \Big(\rho \frac{d}{d\rho}\Big)^m\Phi(\rho) = \frac{1}{2\pi i} \int\limits_{c_0-i\infty}^{c_0+i\infty} X^s\zeta(s-m+1)\Gamma(s)\Bigg(\sum_{\ell \in \mathcal{S}} \frac{1}{\ell^{s-m}}\Bigg) ds,
\end{equation*}
for some $c_0>m+1$.

      Theorem \ref{tst} implies that 
    \begin{equation*}
        \sum_{\ell\in \mathcal{S}} \frac{1}{\ell^s}= f(s)\sqrt{\zeta(s)L(s,\chi)},
    \end{equation*}
    where $\chi$ is the non-principal Dirichlet character modulo 4, and 
    \begin{equation*}
        f(s)=(1-2^{-s})^{-1/2}\prod_{p \equiv 3 \Mod 4} (1-p^{-2s})^{-1/2},
    \end{equation*}
which converges for $\Re( s)>\frac{1}{2}$.

Defining the  function $$g(s):=f(s)\sqrt{(s-1)\zeta(s)L(s,\chi_1)},$$  we can then write for some $c_0>m+1$ 
\begin{align*}
    \Big(\rho \frac{d}{d\rho}\Big)^m\Phi(\rho) & =  \frac{1}{2\pi i} \int\limits_{c_0-iT}^{c_0+iT} X^s\zeta(s-m+1)\Gamma(s) \frac{g(s-m)}{\sqrt{s-m-1}} ds+O\Big(\frac{X^{m+1+\varepsilon}}{e^T}\Big)  \quad  \\ 
    & =  \frac{1}{2\pi i} \int\limits_{\mathcal{C}} X^s\zeta(s-m+1)\Gamma(s) \frac{g(s-m)}{\sqrt{s-m-1}} ds+O\Big(X^{m+1-c}\log T\Big)  \\
    & =  \frac{1}{2\pi i} \int\limits_{\mathcal{C}} X^s\zeta(s-m+1)\Gamma(s) \frac{g(s-m)}{\sqrt{s-m-1}} ds+O\Bigg(\frac{X^{m+1}}{(\log X)^{3/2}}\Bigg), \\
\end{align*}
where $\mathcal{C}$ is the contour running from $m+1-c-i\delta$ along a straight line to $m+1-i\delta$, then along the semicircle $m+1+\delta e^{i\theta}$, $-\frac{\pi}{2}\leq \theta \leq \frac{\pi}{2}$, and finally along a straight line to $m+1-c+i\delta$. Here we have taken  $c= \frac{c'}{\log T}$ for some constant $c'>0$ so that the contour $\mathcal{C}$ is inside the zero free region for $\zeta(s-m)$ up to height $T=\exp(\sqrt{\log X})$, and we take $\delta= \frac{1}{\log X}$. 

Since $g(s-m)$ has a removable fractional singularity in a neighbourhood of $m+1$, we see that

\begin{align*}
         \frac{1}{2\pi i} \int\limits_{\mathcal{C}} & X^s\zeta(s-m+1)\Gamma(s) \frac{g(s-m)}{\sqrt{s-m-1}} ds \\
         &=  \frac{1}{2\pi i} \int\limits_{\mathcal{C}} X^s\zeta(s-m+1)\Gamma(s) \frac{1}{\sqrt{s-m-1}}\Big(g(1)+O(|s-(m+1)|)\Big) ds\\
    &=  \frac{1}{2\pi i} \int\limits_{\mathcal{C}} X^s\zeta(s-m+1)\Gamma(s) \frac{1}{\sqrt{s-m-1}}\Big(K\sqrt{\pi}+O(|s-(m+1)|)\Big) ds.
\end{align*}
The choice of $\delta=\frac{1}{\log X}$ implies that 

\begin{align*}
   \int\limits_{\mathcal{C}} X^s\zeta(s-m+1)\Gamma(s) \frac{1}{\sqrt{s-m-1}}|s-(m+1)| ds& \ll \int\limits_{\mathcal{C}} X^{\Re s} |s-(m+1)|^{\frac{1}{2}} |ds|\\
   & \ll \frac{X^{m+1}}{(\log X)^{3/2}}.
\end{align*}

Now we note that after a change of variable,
\begin{align*}
    \frac{1}{2\pi i} \int\limits_{\mathcal{C}} X^s\zeta(s-m+1)\Gamma(s) \frac{1}{\sqrt{s-m-1}} ds &= \frac{1}{2\pi i} \int\limits_{\mathcal{C'}} X^{m+1+u}\zeta(2+u)\Gamma(m+1+u) \frac{du}{\sqrt{u}} \\
    & =  \frac{1}{2\pi i} \int\limits_{\mathcal{C'}} X^{m+1+u}\Big(\zeta(2)\Gamma(m+1)+O(u) \Big) \frac{du}{\sqrt{u}},
\end{align*}
where $\mathcal{C'}$ is the contour running from $-c-i\delta$ along a straight line to $-i\delta$, then along the semicircle $\delta e^{i\theta}$, $-\frac{\pi}{2}\leq \theta \leq \frac{\pi}{2}$, and finally along a straight line to $-c+i\delta$. Here we have that 
\begin{align*}
    \frac{1}{2\pi i} \int\limits_{\mathcal{C'}}X^u|u|^{1/2}du&\ll \frac{1}{\sqrt{\log X}} \int\limits_{\mathcal{C'}} |du| \ll \frac{1}{(\log X)^{3/2}},
\end{align*}
and by Theorem \ref{H},
\begin{align*}
    \frac{1}{2\pi i} \int\limits_{\mathcal{C'}}X^u\frac{du}{\sqrt{u}} &=\frac{1}{2\pi i} \int\limits_{\mathcal{H}}X^u \frac{du}{\sqrt{u}}  +O\Big(X^{-c}\Big)\\
    & = \frac{1}{\sqrt{\pi \log X}} +O\Big(\frac{1}{(\log X)^{3/2}}\Big).
\end{align*}
Putting all above together the lemma follows.

\end{proof}

\section{Error Term}\label{error}

Next we bound the contribution of the error term, modelling after Lemma 3.2 in \cite{dt}.

\begin{lemma} \label{minor and major} For $N>1$ we have
    \begin{equation*}
         \bigintsss\limits_{X^{-22/15}\leq |\alpha| \leq \frac{1}{2}} \Psi(\rho e(\alpha)) e(-n\alpha) d\alpha \ll_N X^{-N} \Psi(\rho).
    \end{equation*}
\end{lemma}
\begin{proof} 
We first bound the contribution from $\frac{1}{X}\leq|\alpha|\leq \frac{1}{2}$ by bounding the integrand. By symmetry we may restrict to the interval $[\frac{1}{X}, \frac{1}{2}]$.

Observe that Lemma \ref{Landau} implies that for sufficiently large $R_0$ fixed $$|\mathcal{S}\cap [R,2R]|\gg \frac{R}{\sqrt{\log R}},$$ for $R$ for $R>R_0$.

By Dirichlet's approximation theorem we may write $\alpha \in [\frac{1}{X}, \frac{1}{2}]$  as $\alpha= \frac{a}{q}\pm r$ with $(a,q)=1$, $1\leq q\leq 3R_0$ and $0\leq r \leq \frac{1}{3qR_0}$.
We first look at what can be seen as the \textit{minor arcs}: $\alpha \in [\frac{1}{X}, \frac{1}{2}] $ with $\frac{2}{3qX} \leq r \leq \frac{1}{3qR_0}$. 
Noticing that  
    
    \begin{equation*}
        |1-\rho^{\ell} e(\ell\alpha)|^2=(1-\rho^{\ell})^2+2\rho^{\ell}(1-\cos(2\pi \ell\alpha)),
    \end{equation*}
we obtain, using $\|\theta\|$ as the distance from $\theta$ to the set of integers,
\begin{align*}
   \frac{|\Psi(\rho e(\alpha))|}{\Psi(\rho)}& =  \prod_{ \ell\in \mathcal{S}} \Bigg(1+\frac{4\rho^\ell\sin^2(\pi \ell \alpha)}{(1-\rho^\ell)^2}\Bigg)^{-\frac{1}{2}} \leq    \prod_{ \ell \in \mathcal{S}} \Bigg(1+\frac{16\rho^\ell\|\ell\alpha\|^2}{(1-\rho^\ell)^2}\Bigg)^{-\frac{1}{2}}\\
   & \leq \prod_{\substack{\ell \leq \frac{1}{\log \frac{1}{\rho}}\\ \ell \in \mathcal{S}}} \Bigg(1+\frac{5\|\ell\alpha\|^2}{(1-\rho^\ell)^2}\Bigg)^{-\frac{1}{2}} \leq   \prod_{\substack{\frac{1}{3rq}\leq \ell \leq \frac{2}{3rq}\\ \ell \in \mathcal{S}}} \Bigg(1+\frac{5\|\ell\alpha\|^2X^2}{\ell^2}\Bigg)^{-\frac{1}{2}}.
\end{align*}

For $\frac{1}{3rq}\leq \ell \leq \frac{2}{3rq}$, we have that $\frac{1}{3q}\leq |\ell\alpha - \ell\frac{a}{q}| \leq \frac{2}{3q}$ and so $\|\ell\alpha\|\geq \frac{1}{3q}$. Thus

\begin{equation*}
    \frac{|\Psi(\rho e(\alpha))|}{\Psi(\rho)} \leq \prod_{\substack{\frac{1}{3rq}\leq \ell \leq \frac{2}{3rq}\\ \ell \in \mathcal{S}}} \Bigg(1+\frac{5r^2X^2}{4}\Bigg)^{-\frac{1}{2}}.
\end{equation*}
 If $\frac{2}{3qX}\leq r \leq 1/\sqrt{X}$, using Theorem \ref{Landau} we get 

\begin{align*}
     \frac{|\Psi(\rho e(\alpha))|}{\Psi(\rho)} & \leq \prod_{\substack{\frac{1}{3rq}\leq \ell \leq \frac{2}{3rq}\\ \ell \in \mathcal{S}}} \Bigg(1+\frac{5}{ 9q^2}\Bigg)^{-\frac{1}{2}}  \\\
      & \ll  \Bigg(1+\frac{5}{ 9q^2}\Bigg)^{-\frac{c_0}{rq\sqrt{\log rq}}}  \\
      & \ll e^{-\frac{c_1}{\log \frac{1}{\rho}}}\\
       & \ll X^{-N},
\end{align*}
where the constants $c_0$ and $c_1$ depend only on $R_0$. And if $1/\sqrt{X}< r \leq \frac{1}{3qR_0}$, since  Theorem \ref{Landau} implies that $|\mathcal{S}\cap[\frac{1}{3rq},\frac{2}{3rq}]|\geq 2N$ for fixed $R_0$ and sufficiently large $n$, it follows that

\begin{equation*}
    \frac{|\Psi(\rho e(\alpha))|}{\Psi(\rho)} \ll\Bigg(1+\frac{5}{ 4 \log \frac{1}{\rho}}\Bigg)^{N}  \ll X^{-N}.
\end{equation*}

Now we look at what can be regarded as the \textit{nonprincipal major arcs}: $\alpha \in [\frac{1}{X}, \frac{1}{2}]$ with $0\leq r \leq \frac{2}{3qX}$ and $q\geq 2$. 
Let $\{\ell_{q,i}\}_{i=1}^N \subset \mathcal{S}\setminus q\mathbb{N}$ ($2\leq q \leq 3R_0$), which can be chosen due to the infinitude of the set in question. Observe that $\frac{1}{X}\leq \frac{1}{\ell_{q,i}}$ for large $n$, 
so $||\ell_{q,i}\alpha ||\geq \frac{1}{3q}$ for all $\alpha \in [\frac{1}{X}, \frac{1}{2}]$ belonging to the major arcs whose Dirichlet approximation is a fraction with $q$ as denominator. Bounding the Euler products as before, we obtain
\begin{align*}
    \frac{|\Psi(\rho e(\alpha))|}{\Psi(\rho)}  & \leq \prod_{i=1}^N \Bigg(1+\frac{5\|\ell_{q,i}\alpha\|^2X^2}{\ell_{q,i}^2}\Bigg)^{-\frac{1}{2}}\\
    & \leq \prod_{i=1}^N \Bigg(1+\frac{5X^2}{9q^2\ell_{q,i}^2}\Bigg)^{-\frac{1}{2}}\\
    & \ll X^{-N}.
\end{align*}

It only remains to bound the integrand for $X^{-22/15} \leq  \alpha \leq \frac{1}{X}$. We bound the Euler products as before. Noticing that $\|\ell\alpha\|=\ell\alpha $ \ for \ $ \ell<\frac{X}{2}$, and that $|\mathcal{S}\cap[1,\frac{X}{2}]|\gg X/\sqrt{\log X}$ from Theorem \ref{Landau},

\begin{align*}
    \frac{|\Psi(\rho e(\alpha))|}{\Psi(\rho)}& \leq \prod_{ \ell\in \mathcal{S}} \Bigg(1+\frac{9\|\ell\alpha\|^2X^2}{\ell^2}\Bigg)^{-\frac{1}{2}}\\
    & \leq \prod_{\substack{ \ell\leq \frac{1}{2\log \frac{1}{\rho}}\\ \ell \in \mathcal{S}}} \Bigg(1+9 X^{2-44/15}\Bigg)^{-\frac{1}{2}}\\
    & \ll e^{-c_2 X^{2-44/15+1+\varepsilon}}\\
    & \ll_N X^{-N}.
\end{align*}
  Hence the lemma follows.
\end{proof}

\section{Proof of the Main Theorem}\label{proofmain}

We are now ready to proceed with the proof of our main theorem. At last, the Saddle Point Method comes into play: we estimate the integral over what can be regarded as the principal major arc, written in the form of an exponential integral, using the Taylor expansion of $\Phi(\rho)$ around the saddle point $\rho$, whose terms' sizes are given by Lemma \ref{deriv} as anticipated.

\begin{proof}[Proof of Theorem \ref{main}]
We change our notation as suggested in \cite{dt} by defining
\begin{equation*}
    T(s):=\Phi(e^{-s}),
\end{equation*}
to make the proof a bit simpler.

Recalling we set $\rho=e^{-1/X}$ in Lemma \ref{deriv}, we observe $T(1/X)=\Phi(\rho)$, and implied by Lemma \ref{minor and major} 
\begin{equation}\label{newp}
    p_\mathcal{S}(n) = \rho^{-n} \int_{-X^{-22/15}}^{X^{-22/15}} \exp\big(T(1/X-2\pi i \alpha)-2\pi in\alpha\big)  d\alpha+O(\rho^{-n}\Psi(\rho)n^{-B})
\end{equation}
for any constant $B>\frac{1}{2}$, given that $n=x\asymp X^2(\log X)^{-1/2}$ as deduced in Proposition \ref{p}.

Note also that
\begin{equation*}
    T^{(m)}(1/X)= (-1)^m \Big(\rho\frac{d}{d\rho}\Big)^m \Phi(\rho).
\end{equation*}
Hence Lemma \ref{deriv} implies the Taylor expansion of $T(s)$ around $1/X$ takes the form

\begin{equation*}
    T(1/X-2\pi i\alpha)= T(1/X)-2\pi i\alpha T'(1/X)-4\pi^2\alpha^2\frac{T''(1/X)}{2}+O(n^{-1/5})
\end{equation*}
 for $n$ sufficiently large, using that  $n=x\asymp X^2(\log X)^{-1/2}$ once more.

The equation $\rho\Phi'(\rho)=n$ translates to $T'(1/X)=-n$, and the integral in equation (\ref{newp}) can be approximated by a Gaussian integral:

\begin{align*}
    \Psi(\rho) \int_{|\alpha|<X^{-22/15}}& \exp\big(-4\pi^2\alpha^2T''(1/X)/2+O(X^{-2/5})\big) d\alpha\\
    &=\Psi(\rho)\int_{|\alpha|<X^{-22/15}} \exp\big(-4\pi^2\alpha^2T''(1/X)/2\big)\Big(1+O(n^{-1/5})\Big) d\alpha\\
    & = \frac{\Psi(\rho)}{\sqrt{2\pi T''(1/X)}}\Big(1+O(e^{-n^{23/30}})\Big)\Big(1+O(n^{-1/5})\Big).
\end{align*}
Putting this into equation (\ref{newp}) proves the theorem.
\end{proof}

\begin{remark}
    Although one could get an asymptotic expansion  by using more terms of the Taylor expansion of $T(s)$, it will suffice with its first term in our case since the other terms will not make our asymptotic formula more precise for they will be absorbed by the error term in the exponential of $\rho^{-n}\Psi(\rho)$ asymptotically expressed in terms of $n$, that is, of equation (\ref{phi}).
\end{remark}

\section{Proof of the Difference Function}\label{diff}

 Now we obtain the asymptotic formula for the difference function whose proof relies on that of the main theorem.

\begin{proof}[Proof of Theorem \ref{growth}]
    By equation (\ref{Cauchy}) we have that
    \begin{equation*}
        p_\mathcal{S}(n+1)-p_\mathcal{S}(n)= \rho^{-n} \int_{-1/2}^{1/2} \exp\big(\Phi(\rho e(\alpha))-2\pi in\alpha\big) (\rho^{-1}e^{-2\pi i\alpha}-1) d\alpha.
    \end{equation*}
    Note that $|\rho^{-1}e^{-2\pi i\alpha}-1|\leq e^{1/X}+1\leq 4$. From the proof of Theorem \ref{main} we see that the contribution from $|\alpha|>\eta=X^{-3/2}(\log X)^{7/4}$ is
    \begin{equation*}
        \ll \frac{\rho^{-n}\Psi(\rho)}{\sqrt{2\pi \Big\{\Big(\rho\frac{d}{d\rho}\Big)^2 \Phi(\rho)\Big\}}}n^{-B}
    \end{equation*}
    for any positive constant $B>1/2$. On the other hand, when $|\alpha|\leq \eta$ we have 
    \begin{equation*}
        \rho^{-1}e^{-2\pi i\alpha}-1=\exp\Big(\frac{1}{X}-2\pi i\alpha\Big)-1=\frac{1}{X}+O(\eta)=\frac{1}{X}+O\big(X^{-3/2}(\log X)^{7/4}\big).
    \end{equation*}
    From the proof of Theorem \ref{main} we have
    \begin{equation*}
        \int_{-\eta}^\eta \rho^{-n} \exp\big(\Phi(\rho e(\alpha))-2\pi in\alpha\big) d\alpha =\frac{\rho^{-n}\Psi(\rho)}{\sqrt{2 \pi  \Big\{\Big(\rho\frac{d}{d\rho}\Big)^2 \Phi(\rho)\Big\}}}\Big(1+O(n^{-\frac{1}{5}})\Big).
    \end{equation*}
    Recalling equation (\ref{p2}) from Proposition \ref{p}, we have
    
    \begin{align*}
        p_\mathcal{S}(n+1)-p_\mathcal{S}(n) &=\frac{\rho^{-n}\Psi(\rho)}{\sqrt{2 \pi  \Big\{\Big(\rho\frac{d}{d\rho}\Big)^2 \Phi(\rho)\Big\}}}\Big(1+O(n^{-\frac{1}{5}})\Big)\Big(\frac{1}{X}+O\big(X^{-3/2}(\log X)^{7/4}\big)\Big)\\
        &=\frac{\rho^{-n}\log(\frac{1}{\rho})\Psi(\rho)}{\sqrt{2 \pi \Big\{\Big(\rho\frac{d}{d\rho}\Big)^2 \Phi(\rho)\Big\}}}\Big(1+O(n^{-\frac{1}{5}})\Big),
    \end{align*}
    as desired.
\end{proof}

\section*{Acknowledgements}

This work was completed at the University of Mississippi as part of the author’s Ph.D. dissertation research. The author would like to thank his supervisors, Ayla Gafni and Micah Milinovich, for their guidance and support throughout the project.

The author is a Graduate Research Assistant supported by NSF grants DMS-2101912 and DMS-2401461. Part of the work was completed while the author was in residence at UCLA, supported by NSF grant OIA-2229278.

\appendix 
\section{}\label{appendix}

For the convenience of the reader we provide the theorems we have referred to and used throughout  the paper.

\begin{theorem}[Two Square Theorem] \label{tst}
    A positive integer $n$ is the sum of two squares if and only if each of its prime factors $p$ such that $p\equiv 3 \ (\text{mod } 4)$ occurs an even number of times.
\end{theorem}

\begin{proof}
    We refer the reader to \cite{underwood}.
\end{proof}

\begin{theorem}[Landau \cite{L}, 1908] \label{Landau}
Let $S(x)$ be the number of $n$ that can be represented as a sum of two squares. We have that
    \begin{equation*}
        S(x)= K\frac{x}{\sqrt{\log x}}+O\Bigg(\frac{x}{(\log x)^{3/2}}\Bigg),
    \end{equation*}
        where 
    \begin{equation*}
        K=\frac{1}{\sqrt{2}}\prod_{p \equiv \ 3 \Mod 4} \Big(1-\frac{1}{p^2}\Big)^{-1/2} \approx 0.764223653...
            \end{equation*}
    known as the Landau-Ramanujan constant.
\end{theorem}

\begin{proof}
    We refer the reader to Exercise 6.2.21 in Montgomery and Vaughan's \cite{MV}.
\end{proof}

\begin{theorem}[Hankel] \label{H}
    For any complex number $s$,
    \begin{equation*}
        \frac{1}{2\pi i} \int\limits_\mathcal{H} e^zz^{-s} dz= \frac{1}{\Gamma(s)},
    \end{equation*}
    where for $r>0$ we let $\mathcal{H}=\mathcal{H}(r)$ denote the Hankel contour, which consists  of a path that passes from $-\infty-ir$ to $-ir$ along a straight line, and then from $-ir$ to $ir$ along the semicircle $re^{i\theta}$, $-\frac{\pi}{2} \leq \theta \leq \frac{\pi}{2}$, and then from $ir$ to $- \infty +ir $ along a straight line . Here $z^{-s}$ is assumed to have its principal value.
\end{theorem}

\begin{proof}
    We refer the reader to Theorem C.3 in Montgomery and Vaughan's \cite{MV}.
\end{proof}

\end{document}